\documentclass{amsart}

\usepackage[foot]{amsaddr}
\usepackage[utf8]{inputenc}
\usepackage[T1]{fontenc}
\usepackage{amssymb,latexsym}
\usepackage{amsmath}
\usepackage{amsthm}
\usepackage{amssymb}
\usepackage{graphicx}
\usepackage{color}
\usepackage[shortlabels]{enumitem}
\usepackage[colorlinks=true,citecolor={blue},draft=false, urlcolor={black}]{hyperref}
\usepackage{cleveref}

\def\N{\mathbb N}
\def\A{\mathcal A}

\def\uu{\mathbf u}

\def\Der{\mathrm{Der}}

\def\N{\mathbb N}
\def\A{\mathcal A}

\newtheorem{thm}{Theorem}
\newtheorem{theorem}[thm]{Theorem}
\newtheorem{coro}[thm]{Corollary}

\newtheorem{lem}[thm]{Lemma}
\newtheorem{lmm}[thm]{Lemma}

\newtheorem{prop}[thm]{Proposition}
\newtheorem{defi}[thm]{Definition}

\crefname{thm}{theorem}{theorems}
\crefname{theorem}{theorem}{theorems}
\crefname{coro}{corollary}{corollaries}
\crefname{example}{example}{examples}
\crefname{lem}{lemma}{lemmas}
\crefname{lmm}{lemma}{lemmas}
\crefname{claim}{claim}{claims}
\crefname{obs}{observation}{observations}
\crefname{proposition}{proposition}{propositions}
\crefname{prop}{proposition}{propositions}
\crefname{defi}{definition}{definitions}

\newtheorem{remark}[thm]{Remark}

\newtheorem{example}[thm]{Example}

\crefname{example}{example}{examples}

\begin{document}

\title{Fixed points of Sturmian morphisms and their derivated words}


\author{Karel Klouda}

\author{Kateřina Medková}
\thanks{\textit{E-mail:} \url{medkokat@fjfi.cvut.cz} (K. Medková)}

\author{Edita Pelantov\'a}

\author{Štěpán Starosta}

\address[K. Medková, E. Pelantová]{Department of Mathematics, Faculty of Nuclear Sciences and Physical Engineering, Czech Technical University in Prague, Břehová 7, 115 19, Prague 1, Czech Republic}

\address[K. Klouda, Š. Starosta]{Department of Applied Mathematics, Faculty of Information Technology, Czech Technical University in Prague, Thákurova 9, 160 00, Prague 6, Czech Republic}

\date{\today}

\begin{abstract}
Any infinite uniformly recurrent word  ${\bf u}$ can be written as concatenation of a finite number of return words to a chosen prefix $w$ of ${\bf u}$.
Ordering of the return words to $w$ in this concatenation is coded by derivated word $d_{\bf u}(w)$.
In 1998, Durand proved that a fixed point ${\bf u}$ of a primitive morphism has only finitely many derivated words $d_{\bf u}(w)$  and each derivated word  $d_{\bf u}(w)$ is fixed by a primitive morphism as well.
In our article we focus on  Sturmian words fixed by a primitive morphism.
We provide an algorithm which to a given Sturmian morphism $\psi$  lists  the morphisms fixing the derivated words of the Sturmian word ${\bf u} = \psi({\bf u})$.
We provide a sharp upper bound on length of the list.
\newline


\noindent \textit{2000MSC:} 68R15
\end{abstract}

\maketitle

%
%
\section{Introduction}

Sturmian words are probably the most studied object in combinatorics on words.
They are aperiodic words over a binary alphabet having the least factor complexity possible.
Many properties, characterizations and generalizations are known, see for instance \cite{BeSe_Lothaire,Be_survey_corr,BaPeSta2}.

One of their characterizations is in terms of return words to their factors.
Let $\uu = u_0u_1u_2 \cdots$ be a binary infinite word with $u_i \in \{0,1\}$.
Let $w = u_iu_{i+1} \cdots u_{i+n-1}$ be its factor.
The integer $i$ is called an \emph{occurrence} of the factor $w$.
A return word to a factor $w$ is a word $u_iu_{i+1} \cdots u_{j-1}$ with $i$ and $j$ being two consecutive occurrences of $w$ such that $i < j$.
In \cite{Vu}, Vuillon showed that an infinite word $\uu$ is Sturmian if and only if each nonempty factor $w$ has exactly two distinct return words.
A straightforward consequence of this characterization is that if $w$ is a prefix of $\uu$, we may write
\[
\uu = r_{s_0}r_{s_1}r_{s_2}r_{s_3}\cdots
\]
with $s_i \in \{0,1\}$ and $r_0$ and $r_1$ being the two return words to $w$.
The coding of these return words, the word $d_{\uu}(w) = s_0s_1s_2 \cdots$ is called the \emph{derivated word of $\uu$ with respect to $w$}, introduced in \cite{Durand98}.
A simple corollary of the characterization by return words and a result of \cite{Durand98} is that the derivated word $d_{\uu}(w)$ is also a Sturmian word (see \Cref{thm:derstst}).
This simple corollary follows also from other results.
For instance, it follows from \cite{AraBru05}, where the authors investigate the derivated word of a standard Sturmian word and give its precise description.
It also follows from the investigation of a more general setting in~\cite{BuLu}, which may in fact be used to describe derivated words of any episturmian word --- generalized Sturmian words~\cite{GlJu}.

By the main result of \cite{Durand98}, if $\uu$ is a fixed point of a primitive morphism, the set of all derivated words of $\uu$ is finite (the result also follows from \cite{HoZa99}).
In this case, again by \cite{Durand98}, a derivated word itself is a fixed point of a primitive morphism.

In this article we study derivated words of fixed points of primitive Sturmian morphisms.
By the results of \cite{MiSe93}, any primitive Sturmian morphism   may be decomposed using elementary Sturmian morphisms ---  generators of the Sturmian monoid. 
In \Cref{thm:preimages_fi_b,thm:preimages_fi_a}, we describe the relation between the set of derivated words of a Sturmian sequence $\uu$ and the set of derivated words of $\varphi(\uu)$, where $\varphi$ is a generator of the Sturmian monoid.

The main result of our article is an exact description of the morphisms fixing the derivated words $d_{\uu}(w) $ of $\uu$, where $\uu$ is fixed by a Sturmian morphism $\psi$ and $w$ is its prefix.
For this purpose, we introduce an operation $\Delta$ acting on the set of Sturmian morphisms with unique fixed point, see \Cref{def:delta}.
Iterating this operation  we create the desired list of the morphisms as stated in \Cref{thm:main_result_vetsina}.
The Sturmian morphisms with two fixed points are treated separately, see \Cref{lem:revers_of_standard}.

We continue our study by counting the number of derivated words, in particular by counting the  distinct elements in the sequence $\bigl(\Delta^k(\psi)\bigr)_{k \geq 1}$. This number depends on the decomposition of $\psi$ into the generators of the special Sturmian monoid, see below in~\Cref{sec:stu_mor}. 

Using this decomposition, \Cref{cor:standard_sturmian,prop:number_for_aalpha} provide the exact number of derivated words for two specific classes of Sturmian morphisms.

For a general Sturmian morphism $\psi$, \Cref{coro:bounds_on_nr_of_der_words} gives  a sharp upper bound on their number.
The upper bound depends on the number of the elementary morphisms in the decomposition of $\psi$.
In the last section, we give some comments and state open questions.

%
%
\section{Preliminaries}

An \emph{alphabet} $\mathcal{A}$ is a finite set of symbols called \emph{letters}.
A \emph{finite word} of length $n$ over $\mathcal {A}$ is a string $u=u_0u_1\cdots u_{n-1}$, where $u_i \in \mathcal{A}$ for all $i=0,1,\ldots, n-1$.
The \emph{length} of $u$  is denoted by $|u| = n$.  By $|u|_a$ we denote the  number of copies of the letter  $a$  used in $u$, i.e. $|u|_a= \# \{i \in \mathbb{N} \colon  i< n, u_i = a\}$. 
The set of all finite words over  $\mathcal{A}$ together with the operation of concatenation forms a monoid $\mathcal{A}^*$.
Its neutral element is the \emph{empty word} $\varepsilon$ and $\A^+ = \A^* \setminus  \left\{  \varepsilon \right\}$. On this monoid we  work with two operations which preserve the length of words. The \emph{mirror image}  or \emph{reversal} of a  word $u=u_0u_1\cdots u_{n-1} \in \mathcal{A}^*$  is  the  word $\overline{u} = u_{n-1}u_{n-2}\cdots u_{1}u_0$. The \emph{cyclic shift}  of $u$  is  the  word 
\begin{equation}\label{eq:def_of_cyc}
{\rm cyc}(u) = u_{1}u_{2}\cdots u_{n-1}u_0.
\end{equation}

An \emph{infinite word} over $\mathcal {A}$ is a sequence $\uu = u_0u_1u_2\cdots  = \left(u_i\right)_{i\in \mathbb{N}} \in \mathcal{A}^{\mathbb{N}}$ with $u_i \in \mathcal{A}$ for all $i \in \N = \left\{ 0,1,2, \ldots \right \}$.
Bold letters are systematically used to denote infinite words throughout this article.

A finite word $p \in \mathcal{A}^*$ is a \emph{prefix} of $u=u_0u_1\cdots u_{n-1}$ if $p= u_0u_1u_2\cdots u_{k-1}$ for some $k\leq n$, the word $u_k u_{k+1} \cdots u_{n-1}$ is denoted $p^{-1}u$. 
Similarly, $p \in \mathcal{A}^*$ is a prefix of $\uu= u_0u_1u_2\cdots$ if $p=u_0u_1u_2\cdots u_{k-1}$ for some integer $k$. We usually abbreviate  $u_0u_1u_2\cdots u_{k-1} = \uu_{[0,k)}$.

A finite word $w$ is a \emph{factor} of $\uu= u_0u_1u_2\cdots$  if there exists an index $i$ such that $w$ is a prefix of the infinite word $u_iu_{i+1}u_{i+2}\cdots$.
The index $i$ is called an \emph{occurrence} of $w$ in $\uu$.
If each factor of $\uu$ has infinitely many occurrences in $\uu$, the word $\uu$ is \emph{recurrent}.

The \emph{language $\mathcal{L}(\uu)$ of an infinite word $\uu$} is the set of all its factors.
The mapping $\mathcal{C}_{\uu}: \mathbb{N}\mapsto \mathbb{N}$ defined by $\mathcal{C}_{\uu}(n) = \#\{w \in \mathcal{L}(\uu): |w|=n\}$ is called the \emph{factor complexity} of the word $\uu$.

An infinite word $\uu$ is \emph{eventually periodic} if $\uu = wvvvvv\cdots $ for some $v,w \in \mathcal{A}^*$.
If $w$ is the shortest such word possible, we say that $|w|$ is the \emph{preperiod} of $\uu$; if $v$ is the shortest possible, we say that $|v|$ is the \emph{period} of $\uu$.
If $\uu$ is not eventually periodic, it is \emph{aperiodic}.
A factor $w$ of $\uu$ is a \emph{right special} factor if there exist at least two letters $a,b \in \mathcal{A}$ such that $wa, wb$ belong to the language $\mathcal{L}(\uu)$.
A \emph{left special} factor is defined analogously.

An infinite word $\uu$ is eventually periodic if and only if $\mathcal{L}(\uu)$ contains only finitely many right special factors.
Equivalently, $\uu$ is eventually periodic if and only if its factor complexity $\mathcal{C}_{\uu}$ is bounded.
On the other hand, the factor complexity of any aperiodic word satisfies $\mathcal{C}_{\uu}(n)\geq n+1$  for every $n \in \mathbb{N}$.

An infinite word $\uu$ with $ \mathcal{C}_{\uu}(n)= n+1$ for each $n \in \mathbb{N}$ is called Sturmian.
A Sturmian word is \emph{standard} (or \emph{characteristic}) if each of its prefixes is a left special factor.

\subsection{Derivated words} \label{SecDerivatedWords}

Consider a prefix $w$ of an infinite recurrent word $\uu$.
Let $i<j$ be two consecutive occurrences of $w$ in $\uu$.
The string $u_iu_{i+1}\cdots u_{j-1}$ is a \emph{return word} to $w$ in $\uu$.
The set of all return words to $w$ in $\uu$ is denoted by $\mathcal{R}_{\uu}(w)$.
Let us suppose that the set of return words to $w$ is finite, i.e. $\mathcal{R}_{\uu}(w) = \{r_0, r_1, \ldots, r_{k-1}\}$.
The word $\uu$ can be written as unique concatenation of the return words $\uu = r_{s_0}r_{s_1}r_{s_2}\cdots$.
The \emph{derivated} word of $\uu$ with respect to the prefix $w$ is the infinite word
\[
d_{\uu}(w) =  s_0s_1s_2 \cdots
\]
over the alphabet of cardinality $ \#\mathcal{R}_{\uu}(w)=k $.
In his original definition,  Durand \cite{Durand98} fixed the alphabet of the derivated  word to the set $\{0,1, \ldots, k-1\}$.
Moreover, Durand's definition requires that  for $i< j$ the first occurrence of $r_i$  in $\uu$ is less than the first occurrence of $r_j$ in $\uu$.
In particular,  a derivated  word always starts with the letter $0$.
In the article \cite{AraBru05}, where  derivated  words of standard Sturmian words are studied, the authors required that the starting letters of the original word and its derivated word coincide.
For our purposes, we do not need to fix the alphabet of derivated words: two derivated words which differ only by a permutation of letters are identified one with another.

In the sequel, we work only with infinite words which are \emph{uniformly recurrent}, i.e. each prefix $w$ of $\uu$ occurs in $\uu$ infinitely many times and the set $\mathcal{R}_{\uu}(w)$  is finite.
Our aim is to describe the set
\[
\Der(\uu) = \{ d_{\uu}(w) \colon w \text{ is a prefix of }  \uu\}.
\]
Clearly, if a prefix $w$ is not right special, then there exists a unique letter $x$ such that $wx \in \mathcal{L}(\uu)$.
Thus the occurrences of $w$ and $wx$ coincide, $\mathcal{R}_{\uu}(w) = \mathcal{R}_{\uu}(wx)$ and $d_{\uu}(w)$ = $d_{\uu}(wx)$.
If $\uu$ is not eventually periodic, then $w$ is a prefix of a right special prefix of $\uu$.
Therefore for an aperiodic uniformly recurrent  word $\uu$ we have
\[
\Der(\uu) = \{ d_{\uu}(w) \colon w \text{ is a right special prefix of }  \uu \}.
\]

\subsection{Sturmian words}
Any Sturmian word  $\uu$ can be identified with an upper or lower mechanical word.  A mechanical word is described  by two parameters: slope and intercept. The slope  is an irrational number $\gamma \in (0,1)$  and the intercept is a real number $\rho \in [0,1)$.  To define the lower mechanical word ${\bf s}(\gamma, \rho) = \left(s_{n} (\gamma, \rho) \right)_{n \in \N}$ we put  $I_0 =  [0,1-\gamma)$.   The $n^{th}$ letter of ${\bf s}(\gamma, \rho)$ is as follows: 
$$
s_n(\gamma, \rho) =
\begin{cases}
0 & \text{if the number }    \gamma n + \rho \mod 1 \text{ belongs to }  I_0, \\
1 & \text{otherwise.}
\end{cases}
$$
The definition of the upper mechanical word  ${\bf s}'(\gamma, \rho) = \left(s'_n(\gamma, \rho) \right)_{n \in \N}$ is analogous, it just uses the interval $I_0 = (0,1-\gamma]$.  Let us stress that $s_n(\gamma, \rho) \neq s'_n(\gamma, \rho)$ for at most two neighbouring indices  $n$ and $n+1$.  All upper and lower mechanical words with irrational slope are Sturmian and  any Sturmian word equals to a lower or to an upper mechanical word.
Let us stress that one-sided Sturmian words with irrational slope are always uniformly recurrent.
The language of a Sturmian word depends only on $\gamma$.
The number $\gamma$ is in fact the density of the letter $1$, i.e., $\gamma = \lim\limits_{n\to \infty}\tfrac1n \# \left \{i\in \mathbb{N} \colon i < n, s_i(\gamma, \rho) = 1 \right \} $.   Consequently, $1-\gamma$ is the density of the letter $0$.

For any irrational $\gamma \in (0,1)$ there exists a unique  mechanical word ${\bf c}(\gamma)$ with slope $\gamma$ such that  both  $0{\bf c}(\gamma)$  and $1{\bf c}(\gamma)$ are Sturmian.
The word ${\bf c}(\gamma)$ is a standard Sturmian word and  ${\bf c}(\gamma) ={\bf s}(\gamma, \gamma) ={\bf s}'(\gamma, \gamma)$.
Many further properties of Sturmian words can be found in \cite{Lo83,BeSe_Lothaire}.

For our study of derivated words, the following result of Vuillon from \cite{Vu} is important: a word $\uu$  is Sturmian if and only if any prefix of $\uu$ has exactly two return words.
By combining this result with \cite{Durand98}, we obtain an essential observation about derivated words of Sturmian words, which also follows from \cite{AraBru05}.

\begin{theorem} \label{thm:derstst}
If $\uu$ is a Sturmian word and $w$ is a prefix of $\uu$, then its derivated word $d_{\uu}(w)$ is Sturmian as well.
\end{theorem}

\begin{proof}
Set ${\bf v} = d_{\bf u}(w)$.
Let  $p$  be a prefix of ${\bf v}$.
Due to Proposition 2.6 in \cite{Durand98}, there exists a prefix $q$ of $\uu$ such that $d_{\bf v}(p) = d_{\bf u}(q) $.
By Vuillon's characterization of Sturmian words, the word $ d_{\bf u}(q)$ is binary.
It means that  any prefix $p$ of  ${\bf v} $ has  two return words in ${\bf v} $ and so ${\bf v} $ is Sturmian.
\end{proof}

\begin{remark}[Historical]
	The Sturmaian words (sequences) were originally defined by Hedlund and Morse in~\cite{HeMo}. Their definition is more general as they consider also biinfinite words and (in terms of our definition above) rational slopes. Hence their Sturmian words may not be recurrent. For details on the history of definition of Sturmian words see~\cite{Fogg}, especially the historical remark at page~146.
	Interestingly enough, the term \textit{derivated} sequence is also used in~\cite{HeMo}, however, its definition differs from our one (as taken from~\cite{Durand98}): 
	Using again our terminology, their derivated word is a derivated word with respect to a one-letter word in a biinfinite Sturmian word.
\end{remark}

\subsection{Sturmian morphisms} \label{sec:stu_mor}

A \textit{morphism} over $\mathcal{A}^*$ is a mapping $\psi : \mathcal{A}^* \mapsto  \mathcal{A}^*$  such that $\psi(vw) = \psi(v)\psi(w)$ for all $v,w \in \mathcal{A}^*$.
The domain of the  morphism $\psi$ can be naturally extended to $\mathcal{A}^{\mathbb{N}}$ by
$$
\psi(u_0u_1u_2\cdots ) = \psi(u_0)\psi(u_1) \psi(u_2)\cdots .
$$
A morphism  $\psi$ is \emph{primitive} if there exists a positive integer  $k$ such that the letter $a$ occurs in the  word  $\psi^k(b)$ for each pair of letters $a,b \in  \mathcal{A}$.
A \emph{fixed point} of a morphism $\psi$ is an infinite word $\uu$ such that  $\psi(\uu) = \uu$.

A morphism $\psi$ is a \emph{Sturmian morphism} if $\psi(\uu)$ is a Sturmian word for any Sturmian word $\uu$. The set of Sturmian morphisms together with composition forms the so-called  \emph{Sturmian monoid} usually denoted {\it St}.
We work with these four elementary Sturmian morphisms:
$$
	\varphi_a: \begin{cases} 0 \to 0 \\ 1 \to 10 \end{cases} \quad
	\varphi_b: \begin{cases} 0 \to 0 \\ 1 \to 01 \end{cases} \quad
	\varphi_\alpha: \begin{cases} 0 \to 01 \\ 1 \to 1 \end{cases} \quad
	\varphi_\beta: \begin{cases} 0 \to 10 \\ 1 \to 1 \end{cases}
$$
 and  with the monoid $\mathcal{M}$ generated by them, i.e.
$\mathcal{M} = \langle \varphi_a, \varphi_b, \varphi_\alpha, \varphi_\beta \rangle $.
The monoid $\mathcal{M}$ is also called \emph{special Sturmian monoid}.
For a nonempty word $u = u_0\cdots u_{n-1}$ over the alphabet $\{a,b,\alpha,\beta\}$ we put
$$ \varphi_u = \varphi_{u_0} \circ \varphi_{u_1} \circ \cdots \circ \varphi_{u_{n-1}}.
$$

The monoid $\mathcal{M}$ is not free.
It is easy to show that for any $k \in \N$ we have  	$$\varphi_{\alpha a^k\beta } = \varphi_{\beta b^k\alpha}\quad \text{ and } \quad \varphi_{a\alpha^kb} = \varphi_{b\beta^ka}.$$
We can equivalently say that the following rewriting rules hold on the set of words from $\{a,b,\alpha,\beta\}^*$:
\begin{equation}\label{eq:relations}
  \alpha a^k\beta = \beta b^k\alpha \quad \text{ and } \quad  a\alpha^kb = b\beta^ka\qquad \text{ for any $k \in \N$ }.
\end{equation}

In~\cite{See91}, the author reveals a presentation of the Sturmian monoid  which includes the special Sturmian monoid  $\mathcal{M}= \langle \varphi_a, \varphi_b,\varphi_\alpha, \varphi_\beta\rangle$.
A presentation of the special Sturmian monoid follows from this result.
It is also given explicitly in~\cite{ReKa07}:
\begin{thm}\label{thm:relations}
Let $w,v\in \{a,b,\alpha,\beta\}^*$.
The morphism $\varphi_w$ is equal to $\varphi_{v}$ if and only if the word $v$ can be obtained from $w$ by applying the rewriting rules~\eqref{eq:relations}.
\end{thm}

Note that the presentation of a generalization of the Sturmian monoid, the so-called \emph{episturmian monoid}, is also known, see~\cite{Ri03}.
The next lemma summarizes several simple and well-known  properties of Sturmian morphisms we exploit in the sequel.

\begin{lem}\label{lem:properties_of_sturm_morph}
	Let $w \in \{a,b,\alpha,\beta\}^+$.
	\begin{enumerate}[(i)]
		\item The morphism $\varphi_w$  is primitive if and only if  $w$ contains at least one Greek letter $\alpha$ or $\beta$ and at least one Latin letter $a$ or $b$.
		\item If $\varphi_w$ is primitive, then each of its fixed points is aperiodic and uniformly recurrent.
		\item If $\varphi_w$ is primitive, then it has two fixed points if and only if  $w$ belongs to  $\{a,\alpha\}^* $.
	\end{enumerate}
\end{lem}

For  $w \in  \{a,b,\alpha,\beta\}^*$ the rules~\eqref{eq:relations} preserve  positions in $w$ occupied  by Latin letters $ \{a,b\}$  and  positions occupied by Greek letters $ \{\alpha,\beta\}$.
We define that $a<b$ and $\alpha < \beta$ which allows the following definition.

\begin{defi} Let $w \in  \{a,b,\alpha,\beta\}^*$.
The lexicographically greatest word in $  \{a,b,\alpha,\beta\}^*$ which  can be obtained from $w$ by application of rewriting rules \eqref{eq:relations} is denoted $N(w)$.
If  $\psi = \varphi_w$,  then the word $N(w)$ is the \emph{normalized name} of the morphism $\psi$  and it is also denoted by $N(\psi) = N(w)$.
\end{defi}
The next lemma is a direct consequence of Theorem \ref{thm:relations}.
\begin{lem}\label{lem:normalized_words}
Let $w  \in \{a, b, \alpha, \beta\}^*$.
We have $w = N(w)$ if and only if $w$ does not contain $\alpha a^k\beta$ or $a\alpha^kb$ as a factor for any $k \in \mathbb{N}$.
In particular,  if $w \in \{ a, b, \alpha, \beta\}^* \setminus \{a, \alpha\}^*$, the  normalized name $N(w)$ has prefix either $a^i\beta$ or $\alpha^ib$ for some $i \in \mathbb{N}$.
\end{lem}

\begin{example} Since $\psi =  \varphi_a\varphi_b\varphi_\alpha\varphi_b =   \varphi_b\varphi_a\varphi_\alpha\varphi_b =   \varphi_b\varphi_b\varphi_\beta\varphi_a$, the normalized name  of $\psi$ is $N(\psi) = bb\beta a$.
\end{example}

The morphism  $E: 0\to 1, 1\to 0$ which exchanges letters in words over $\{0,1\}$  cannot change the factor complexity of an infinite word.
Thus, $E$ is clearly a Sturmian morphism.
But $E$ does not belong to the monoid  $\mathcal{M} = \langle  \varphi_a,\varphi_b,\varphi_\alpha,\varphi_\beta\rangle  $.
In fact, $E$  is the only missing morphism.
More precisely, any Sturmian morphism $\psi$ either belongs to $\mathcal{M}$ or  $\psi = \eta\circ E$, where $\eta \in \mathcal{M}$ (see~\cite{MiSe93}).
To generate the whole monoid of Sturmian morphisms $St$,  one needs only three morphisms, say  $E$, $\varphi_a$ and $\varphi_b$ (see~\cite{Lo83}).
We have
\begin{equation}\label{zamena}
\varphi_\alpha = E\varphi_aE  \quad \text{ and } \quad \varphi_\beta = E\varphi_bE .
\end{equation}

 Our aim is to study derivated words of fixed points of Sturmian morphisms.
 If $\uu$ is a fixed point of $\psi$, it is also a fixed point of $\psi^2$.
 Due to~\eqref{zamena},  the square $\psi^2$ always belongs to $\mathcal{M}$.
 To illustrate why this is true, assume, e.g., that $\psi \in St = \langle E, \varphi_a, \varphi_b \rangle$ equals $\psi = \varphi_a E \varphi_b \varphi_a$. 
 Using~\eqref{zamena} and the fact that $E^2$ is the identity morphism, we have
 \[
 	\psi = \varphi_a E \varphi_b E E \varphi_a E E = \varphi_a  \varphi_\beta \varphi_\alpha E
 \]
 and hence
 \[
 	\psi^2 = \varphi_a \varphi_\beta \varphi_\alpha E \varphi_a \varphi_\beta \varphi_\alpha E = \varphi_a \varphi_\beta \varphi_\alpha E \varphi_a E E\varphi_\beta E E\varphi_\alpha E = \varphi_a \varphi_\beta \varphi_\alpha \varphi_\alpha \varphi_b \varphi_a \in \mathcal{M}.
 \]
 Therefore we may restrict ourselves to  fixed points of morphisms from the  special Sturmian monoid $\mathcal{M}$.
 Note that this would not be true if we consider only the morphisms from $\langle \varphi_a, \varphi_b \rangle$, see also \Cref{lem:properties_of_sturm_morph}.

\begin{example}\label{fibonacci1} The Fibonacci word is the fixed point of the morphism $\tau: 0\to 01, 1\to 0$.
The morphism $\tau$ is Sturmian, but $\tau  \notin \mathcal{M}$.
We see that $\tau =  \varphi_b\circ E$ and by the relations  \eqref{zamena}  we have  $\tau^2 =\varphi_b\varphi_\beta $.
\end{example}

\begin{remark}\label{nerozlisuj} Two infinite words $\uu$ and $E(\uu)$ over the alphabet $\{0,1\}$ coincide  up to a permutation of the letters $0$ and $1$.
If a word $\uu$ is a fixed point of a morphism  $\varphi_w$, then $E(\uu)$ is a fixed point of the morphism  $E\circ \varphi_w\circ E = \varphi_v$ for some $v$.
By  \eqref{zamena}, the word  $v$ is obtained from $w$ by exchange of letters   $a\leftrightarrow \alpha$ and $b\leftrightarrow\beta$.
Therefore we introduce  the following morphism $F: \{a,b,\alpha,\beta\}^* \mapsto \{a,b,\alpha,\beta\}^* $ by 
\begin{equation}\label{eq:definition_of_F}
F(a)=\alpha, \quad F(\alpha)=a,\quad F(b)=\beta, \quad F(\beta)=b.
\end{equation}
This notation  enables  us to formulate two useful facts   on composition of  $E$ with  morphisms from $\mathcal{M}$. Namely,
\begin{equation}\label{skladani} E\circ \varphi_w\circ E = \varphi_{F(w)} \qquad \text{ and} \qquad (\varphi_w\circ E)^2 = \varphi_{wF(w)}\,.
\end{equation}

\end{remark}

Later on we will need the following statement on the morphism $F$.
First we recall two classical results on word equations:
\begin{lmm}[\cite{LynSch62}] \label{lem:Lyndon1}
	Let $y \in \A^*$ and $x,z \in \A^+$. Then $xy = yz$ if and only if there are $u, v \in \A^*$ and $\ell \in \N$ such that $x = uv, z = vu$ and $y = (uv)^\ell u$.
\end{lmm}
\begin{lmm}[\cite{LynSch62}] \label{lem:Lyndon2}
Let $x, y \in \A^+$. The following three conditions are equivalent:
\begin{enumerate}[(i)]
\item $xy  = yx$;
\item There exist integers $i, j > 0$ such that $x^i = y^j$;
\item There exist $z \in \A^+$ and integers $p, q > 0$ such that $x = z^p$ and $y = z^q$.
\end{enumerate}
\end{lmm}
With these two lemmas we prove the following result on word equations involving the morphism $F$.
Note that this result is within the general setting considered in~\cite{JoFlMaNo}, however we give an explicit solution of cases that we need later.
\begin{lmm}\label{lem:words_equations_F}
	Let $z$ and $p$ be nonempty words from $\{a,b,\alpha,\beta\}^+$.
	\begin{enumerate}[(i)]
		\item If $zp = F(p)F(z)$, then there is $x \in \{a,b,\alpha,\beta\}^+$ such that

 \centerline{$z = x\bigl(F(x)x\bigr)^i$ \ \  and \ \  $p = \bigl(F(x)x\bigr)^jF(x)$ \ \ \  for some $i,j \in \N$.}
\medskip

		\item If $zp = pF(z)$, then there is $x \in \{a,b,\alpha,\beta\}^+$ such that

 \centerline{ $z = \bigl(F(x)x\bigr)^i$ \ \ and \ \ $p =\bigl(F(x)x\bigr)^jF(x)$ \ \ \  for some $i,j \in \N$.}
	\end{enumerate}
\end{lmm}

\begin{proof}
	We prove Item $(i)$ by induction on $|zp| \geq 2$.
	If $|z| = |p|$, then $z = F(p)$ and the statement is true for $x = z$ and $i = j = 0$.

	Assume $|z| > |p|$ (the case of $|z| < |p|$ is analogous).
	There must be a nonempty word $q$ such that $z = F(p)q$ and this yields $qp = F(z) = pF(q)$.
	By \Cref{lem:Lyndon1} there are words $u$ and $v$ and $\ell \in \N$ such that $q = uv, p = (uv)^\ell u$ and $F(q) = vu$.
	This implies that $vu = F(u)F(v)$ and  we can apply the induction hypothesis as
	 $|uv| < |pz|$.
	 Therefore, there are $x$ and $s,r \in \N$ such that $v = x\bigl(F(x)x\bigr)^s$ and $u = \bigl(F(x)x\bigr)^tF(x)$.
	Putting this altogether we obtain
	\begin{align*}
		q & = uv =\bigl(F(x)x\bigr)^tF(x)x\bigl(F(x)x\bigr)^s = \bigl(F(x)x\bigr)^{t+s+1}, \\
		p & = (uv)^\ell u = \bigl(F(x)x\bigr)^jF(x), \quad \text{with} \ j={\ell(t+s+1) + t} \\
		z & = F(p)q =x \bigl(F(x)x\bigr)^i,  \quad \text{with} \ i=\ell(t+s+1) + 2t + s + 1.
	\end{align*}

To prove Item $(ii)$,  we apply  \Cref{lem:Lyndon1} on $zp = pF(z)$.
We have $z = uv, F(z) = vu$ and $p = (uv)^\ell u$ for some words $u$ and $v$ and $\ell \in \N$.
	It follows that $vu = F(u)F(v)$ and so, by Item $(i)$, there is $x$ such that $v = x\bigl(F(x)x\bigr)^i$ and $u = \bigl(F(x)x\bigr)^jF(x)$ for some $i,j \in \N$.
	Using all these equations we finish the proof by stating that
	\[
		z = uv = \bigl(F(x)x\bigr)^jF(x)x\bigl(F(x)x\bigr)^i = \bigl(F(x)x\bigr)^{j+i+1} \quad \text{and} \quad p = (uv)^\ell u = \bigl(F(x)x\bigr)^{\ell(j+i+1) + j}F(x). \qedhere
	\]
\end{proof}

%
%
\section{Derivated words of  Sturmian preimages} \label{sec:der_wo_St_preim}

In this section we study relations between derivated  words of a Sturmian word and derivated words of its preimage under one of the morphisms $\varphi_a, \varphi_b, \varphi_\alpha$  and $\varphi_\beta$.
We prove that the set of all derivated words of these two infinite words coincide up to at most one derivated word, see Theorems~\ref{thm:preimages_fi_b} and~\ref{thm:preimages_fi_a}.
This will be crucial fact for proving the main results of this paper. 
Because of~\eqref{zamena}, the roles of $\varphi_a$  and $\varphi_\alpha$  and, analogously, the roles of $\varphi_b$ and $\varphi_\beta$  are symmetric.
Therefore we can restrict the statements and proofs in this section to the morphisms $\varphi_a$ and $\varphi_b$ with no loss of generality.
Again we use results from~\cite{Lo83}, in particular this slightly modified Proposition 2.3.2:
\begin{prop}[\cite{Lo83}]
	Let $\bf x$ be an infinite word.
	\begin{enumerate}[(i)]
		\item If $\varphi_b(\bf x)$ is Sturmian, then $\bf x$ is Sturmian.
		\item If $\varphi_a(\bf x)$ is Sturmian  and $\bf x$ starts with the letter $1$, then $\bf x$ is Sturmian.
	\end{enumerate}
\end{prop}

\begin{thm}\label{thm:preimages_fi_b}
	Let $\uu$ and $\uu'$ be Sturmian words such that $\uu = \varphi_{b}(\uu')$.
	Then the sets of their derivated words satisfy
	\[
		\Der(\uu) = \Der(\uu') \cup \{ \uu'\}\,.
	\]
\end{thm}
The proof of the previous theorem is split into two parts:  In  Proposition~\ref{prop:der_of_images_fi_b},   Item (i)  says  $ \{ \uu'\} \subset \Der(\uu)$ and  Item (ii)  says $ \Der(\uu)  \subset \Der(\uu') \cup \{ \uu'\}.$  Proposition~\ref{prop:der_of_preimage_fi_b} says $\Der(\uu')  \subset \Der(\uu) $.
Proofs of these propositions use the following simple property of the injective morphism $\varphi_{b}$.
\begin{lem}\label{lem:unique_preimage_under_fi_b}
	Let $\uu = \varphi_{b}(\uu')$ be a Sturmian word.
	If $p0 \in \mathcal{L}(\uu)$  and $0$ is a prefix of $p$, then there exists a unique factor $p'\in \mathcal{L}(\uu')$  such that $p0= \varphi_b(p')0$.
\end{lem}

\begin{prop} \label{prop:der_of_images_fi_b}
	Let $\uu$ and $\uu'$ be Sturmian words such that $\uu = \varphi_{b}(\uu')$ and let $w$ be a prefix of $\uu$.
	\begin{enumerate}[(i)]
		\item If $|w|=1$, then  $d_{\uu}(w) = \uu'$ (up to a permutation of letters). \label{it:der_of_images_fi_b_1}
		\item If $|w|>1$, then there exists a prefix $w'$ of $\uu'$ such that $|w'|<|w|$ and  $d_{\uu}(w) = d_{\uu'}(w')$ (up to a permutation of letters). \label{it:der_of_images_fi_b_2}
		Moreover, if $w$ is right special, $w'$ is right special as well.
	\end{enumerate}
\end{prop}

\begin{proof}
	Since $\varphi_b(0) =0$ and $\varphi_b(1)=01$, the word $\uu = \varphi_{b}(\uu')$ has a prefix $0$ and the letter $1$ is in $\uu$ separated by blocks $0^k$ with $k\geq 1$.
	Therefore, the two  return words in $\uu$  to the word $w=0$  are $r_0=0$ and $r_1=01$.
	We may write $\uu = r_{s_0} r_{s_1}  r_{s_2}  \cdots$, where $r_{s_j} \in \{r_0,r_1\}$ and thus $ d_{\uu}(w) = s_0s_1s_2\cdots$.
	Since $r_0=\varphi_b(0)$ and $r_1=\varphi_b(1)$, we obtain also $\varphi_b(\uu') = \uu = \varphi_b(s_0)  \varphi_b(s_1) \varphi_b(s_2)  \cdots = \varphi_b(s_0s_1s_2\cdots)$.
	The statement in~\ref{it:der_of_images_fi_b_1} now follows from injectivity of $\varphi_b$.

	Now suppose that the prefix $w$ of $\uu$ is of length $>1$.
	As explained earlier, it suffices to consider right special prefixes.
	Since the letter $1$ is always followed by $0$, each right special factor must end in~$0$.
	So the first and the last letter of~$w$ is~$0$, hence by Lemma~\ref{lem:unique_preimage_under_fi_b} there is a unique prefix $w'$ of $\uu'$ such that $\varphi_b(w')0 = w$.
	Let $r_0$ and $r_1$ be the two return words to $w$ and let $\uu = r_{s_0}r_{s_1}r_{s_2}\cdots$.
	Since the first letter of both $r_0$ and $r_1$ is~$0$, there are uniquely given $r'_0$ and $r'_1$ such that  $r_0=\varphi_b(r'_0)$ and $r_1=\varphi_b(r'_1)$ and $\uu' = r'_{s_0}r'_{s_1}r'_{s_2}\cdots$.

	Clearly $w'$ is a prefix of $r'_{s_j}r'_{s_{j+1}}r'_{s_{j+2}}\cdots$ for all $j \in \N$ and so the number $|r'_{s_0}r'_{s_1}\cdots r'_{s_k}|$ is an occurrence of $w'$ in $\uu'$ for all $k \in \N$.
	Let $i > 0$ be an occurrence of $w'$ in $\uu'$.
	It follows that $\varphi_b\bigl(\uu'_{[0,i)}\bigr)w$  is a prefix of~$\uu$ and $|\varphi_b\bigl(\uu'_{[0,i)}\bigr)|$ is an occurrence of $w$ in $\uu$.
	There must be $j \in \N$ such that $\varphi_b\bigl(\uu'_{[0,i)}\bigr) = r_{s_0}r_{s_1}\cdots r_{s_j}$ and hence, by injectivity of $\varphi_b$, $\uu'_{[0,i)} = r'_{s_0}r'_{s_1}\cdots r'_{s_j}$ and $i = |r'_{s_0}r'_{s_1}\cdots r'_{s_j}|$.

	We have proved that the numbers $0$ and $|r'_{s_0}r'_{s_1}\cdots r'_{s_j}|, j = 0,1,\ldots$, are all occurrences of $w'$ in $\uu'$.
	It follows that $r'_0$ and $r'_1$ are the two return words to $w'$ in $\uu'$ and
	\[
		d_{\uu'}(w') = s_0s_1s_2 \cdots =  d_{\uu}(w).
	\]
 
	Since $w = \varphi_b(w')0$ is a right special factor, we must have that both $\varphi_b(w')00$ and $\varphi_b(w')01$ are factors of $\uu$.
	It follows that both $w'0$ and $w'1$ are factors of $\uu'$ and $w'$ is right special. 
\end{proof}

\begin{prop}\label{prop:der_of_preimage_fi_b}
	Let $\uu$  and $\uu'$ be Sturmian words such that $\uu = \varphi_{b}(\uu')$ and let $w'$ be a nonempty right special prefix of $\uu'$.
	Then $d_{\uu'}(w') = d_{\uu}(w)$, where $w =  \varphi_{b}(w')0$.
\end{prop}

\begin{proof}
	Let $r_0'$ and $r_1'$ be the two return words to $w'$ in $\uu'$ and $d_{\uu'}(w') = s_0s_1s_2\cdots$.
	Put $w=\varphi_b(w')0$,   $r_0=\varphi_b(r'_0)$ and $r_1=\varphi_b(r'_1)$.
	We obtain
	\[
		\uu = \varphi_b(\uu') = \varphi_b(r'_{s_0}r'_{s_1}r'_{s_2}\cdots ) = r_{s_0}r_{s_1}r_{s_2}\cdots
	\]
	Clearly,  $w$ is prefix of $r_{s_k}r_{s_{k+1}}r_{s_{k+2}}\cdots$ for all $k \in \N$ and $|r_{s_0}r_{s_1}\cdots r_{s_j}|$ is an occurrence of $w$ in $\uu$ for all $j \in \N$.

	Assume now $i > 0$ is an occurrence of $w$ in $\uu$.
	This means that $\uu_{[0,i)}w$ is a prefix of $\uu$ and hence, by Lemma~\ref{lem:unique_preimage_under_fi_b} (note that $w$ begins with $0$), there must be $p'$ a prefix of $\uu'$ such that $\varphi_b(p') = \uu_{[0,i)}$ and $p'w'$ is a prefix of $\uu'$.
	Since $|p'|$ is an occurrence of $w'$ in $\uu'$, there is $j \in \N$ such that $p' = r'_{s_0}r'_{s_1}\cdots r'_{s_j}$.
	It follows that
	\[
		\uu_{[0,i)} = \varphi_b(r'_{s_0}r'_{s_1}\cdots r'_{s_j}) = r_{s_0}r_{s_1}\cdots r_{s_j}
	\]
	and $i = |r_{s_0}r_{s_1}\cdots r_{s_j}|$.

	So, again as in the previous proof, we have shown that the numbers $0$ and $|r_{s_0}r_{s_1}\cdots r_{s_j}|, j = 0,1,\ldots$, are all occurrences of $w$ in $\uu$.
	It follows that $r_0$ and $r_1$ are the two return words to $w$ in $\uu$ and
	\[
		d_{\uu}(w) = s_0s_1s_2 \cdots =  d_{\uu'}(w'). \qedhere
	\]
\end{proof}

\begin{thm}\label{thm:preimages_fi_a}
Let $\uu$ and $\uu'$ be Sturmian words such that $\uu$ starts with the letter $1$ and $\uu = \varphi_{a}(\uu')$.
Then $\uu'$ starts with $1$ and the sets of their derivated words coincide, i.e.,
\[
	\Der(\uu) = \Der(\uu')\,.
\]
In particular, for any prefix $w$ of $\uu$ there exists a prefix $w'$ of $\uu'$ such that $|w'| \leq |w|$ and  $d_{\uu}(w) = d_{\uu'}(w')$ (up to a permutation of letters).
Moreover, if $w$ is right special, $w'$ is right special as well.
\end{thm}

\begin{proof}
	The morphisms $\varphi_a$ and $\varphi_b$ are conjugate, that is, $0\varphi_a(x)=\varphi_b(x)0$ for each word $x$.
	This means that for any prefix $u'_0u'_1u'_2 \cdots u'_k$ of $\uu'$ we have $0\varphi_a(u'_0u'_1u'_2 \cdots u'_k) = \varphi_b(u'_0u'_1u'_2 \cdots u'_k)0$.
	As this holds true for each $k$, we obtain $0\uu = 0\varphi_a(\uu') =\varphi_b(\uu')$.

	Denote ${\bf v} = v_0v_1v_2\cdots  = 0u_0u_1u_2 \cdots$.
	We have $v_i = u_{i-1}$ for each $i\geq 1$.
	Let $w$ be a nonempty prefix of $\uu$ and $(i_n)$ be the increasing sequence of its occurrences in $\uu$. Note that  $w$ starts with the letter $1$.
	This letter is in $\uu$ surrounded by $0$'s.
	Thus the sequence $(i_n) $ is also the sequence of occurrences of $0w$ in ${\bf v}$ and thus $d_{\bf v}({0w}) = d_{\bf u}({w})$.
	It follows that
	\[
		\Der(\uu)=\{d_{\bf v}(v) \colon v \text{ is a prefix of $\bf v$ and } |v|>1\}\,.
	\]
	We finish the proof by applying Theorem \ref{thm:preimages_fi_b} and Proposition \ref{prop:der_of_images_fi_b} to the word ${\bf v} =\varphi_b(\uu')$.
\end{proof}

The only case which is not treated by Theorems \ref{thm:preimages_fi_b} and \ref{thm:preimages_fi_a}, namely the case when $\uu = \varphi_a(\uu')$ and $\uu$ begins with $0$, can be translated into one of the previous cases.

\begin{lem}\label{lem:prevod_fi_a_na_fi_b}
	Let $\uu$ be a Sturmian word such that $\uu$ starts with the letter $0$ and $ \uu = \varphi_a(\uu')$ for some word $\uu'$.
	Then there exists a Sturmian word $\bf v$ such that $\uu' = 0{\bf v}$ and $\uu = \varphi_b({\bf v})$.
\end{lem}

\begin{proof}
	Since $\uu$ starts with $0$, the form of $\varphi_a$ implies that $\uu' = 0{\bf v}$ for some Sturmian word $\bf v$.
	As $0\varphi_a(x)=\varphi_b(x)0$ for each word $x$, we have
	\[
		\uu = \varphi_a(\uu') = \varphi_a(0{\bf v}) = 0 \varphi_a({\bf v}) = \varphi_b({\bf v}). \qedhere
	\]
\end{proof}
To sum up the results of this section, let us assume we have a sequence of Sturmian words $\uu_0, \uu_1, \uu_2, \ldots$ such that $\uu = \uu_0$ and for every $i \in \N$ one of the following is true:
\begin{enumerate}[(i)]
	\item $\uu_i = \varphi_b(\uu_{i+1})$ or $\uu_i = \varphi_\beta(\uu_{i+1})$, \label{it:seq_pos_1}
	\item $\uu_i$ begins with $1$ and $\uu_i = \varphi_a(\uu_{i+1})$, \label{it:seq_pos_2}
	\item $\uu_i$ begins with $0$ and $\uu_i = \varphi_\alpha(\uu_{i+1})$. \label{it:seq_pos_3}
\end{enumerate}
If \ref{it:seq_pos_1} holds for $\uu_i$, then by Theorem~\ref{thm:preimages_fi_b}
\[
	\mathrm{Der}(\uu_i) = \mathrm{Der}(\uu_{i+1}) \cup \{\uu_{i+1}\},
\]
moreover, $\uu_{i+1}$ is the derivated word of the first letter of $\uu_i$.
This first letter is also the shortest right special prefix.
If \ref{it:seq_pos_2} or \ref{it:seq_pos_3} holds for $\uu_i$, then by Theorem~\ref{thm:preimages_fi_a}
\[
	\mathrm{Der}(\uu_i) = \mathrm{Der}(\uu_{i+1}).
\]

The crucial assumption, namely the existence of the above described sequence $(\uu_k)_{k\geq 0}$, is guaranteed by the well-known fact on the desubstitution of Sturmian words (see, e.g.,~\cite{JuPi} and~\cite{HeMo} and also \Cref{lem:prevod_fi_a_na_fi_b}).
Here we formulate this fact as the following theorem:
\begin{thm}[\cite{JuPi}, \cite{HeMo}]
An infinite binary word $\uu$ is Sturmian if and only if there exists an infinite word ${\bf w}=w_0w_1w_2\cdots$ over the alphabet $\{a,b,\alpha, \beta\}$ and an infinite sequence $(\uu_i)_{i\geq 0}$,  such that  $\uu = \uu_0$ and 
$\uu_{i} = \varphi_{w_i}(\uu_{i + 1})$ for all $i \in \mathbb{N}$. 
 \end{thm}  

In the following section  we work only with   the sequence $(\uu_i)_{i \geq 0}$ corresponding to a fixed point $\uu$ of  a Sturmian morphism $\psi$.  The next lemma provides us a simple technical tool for a description of the elements $\uu_i$   as fixed points of some  Sturmian morphisms.

\begin{lem}\label{lem:rotace_morfizmu}
	Let $\xi$ and $\eta$ be Sturmian morphisms and $\uu = \bigl(\xi \circ \eta\bigr) (\uu)$.
	If $\uu = \xi(\uu')$ for some $\uu'$, then $\uu'$ is the fixed point of the morphism $\eta \circ \xi$, i.e. $\uu' = \bigl( \eta \circ \xi\bigr)(\uu')$.
\end{lem}

\begin{proof}
	For any Sturmian morphism $\xi$, the equation $\xi(\bf x) = \xi (\bf y)$ implies that ${\bf x}= { \bf y}$.
	We deduce that
	\[
	\xi(\uu') = \uu = \bigl(\xi \circ \eta\bigr) (\uu) = \bigl(\xi \circ \eta\bigr) \bigl(\xi (\uu')\bigr) = \bigl( \xi \circ \eta\circ \xi\bigr) (\uu')\,,
	\]
	and so $\uu' = \bigl(\eta \circ \xi\bigr)(\uu')$.
\end{proof}

%
%
\section{Derivated words of fixed points of Sturmian morphisms} \label{sec:der_wo_of_fixed_points}

Let $\uu$ be an fixed point of a primitive Sturmian morphism (note that if the morphism is primitive, all its fixed points are aperiodic).
It is known due to Durand~\cite{Durand98} that the set $\mathrm{Der}(\uu)$ is finite (as the morphism is primitive).
Put
\[
	\mathrm{Der}(\uu) = \{\mathbf{x}_1, \mathbf{x}_2, \ldots, \mathbf{x}_\ell\}.
\]
Our main result is an algorithm that returns a list of Sturmian morphisms $\psi_1, \psi_2, \ldots, \psi_\ell$ such that $\mathbf{x}_i$ is a fixed point of $\psi_i$ (up to a permutation of letters) for all $i$ such that $1 \leq i \leq \ell$.

As we have noticed before, we can restrict ourselves to the morphisms belonging to the monoid $\mathcal{M} = \langle \varphi_a, \varphi_b, \varphi_\alpha, \varphi_\beta \rangle $.
Let us recall (see Lemma~\ref{lem:properties_of_sturm_morph}) that a morphism from $\langle \varphi_a, \varphi_b \rangle $ or from $\langle \varphi_\alpha, \varphi_\beta \rangle $ is not primitive and has no aperiodic fixed point.
Thus we consider only morphisms $\varphi_w$ whose normalized name $w$ contains at least one Latin and one Greek letter.

We will treat two cases separately.
The first one is the case when the morphism $\varphi_w$ has only one fixed point. Lemma~\ref{lem:properties_of_sturm_morph} says that in such a case $w \notin \{a, \alpha\}^*$.
In the second case, when $w \in \{a, \alpha\}^*$, the morphism  $\varphi_w$ has two fixed points.

\subsection{Morphisms with unique fixed point}
Let $\psi \in  \langle\varphi_a, \varphi_b, \varphi_\alpha, \varphi_\beta\rangle$ and $ N(\psi) = w \in \{a,b,\alpha,\beta\}^*\setminus \{a, \alpha\}^*$ be the normalized name of the morphism $\psi$. By  Lemma
\ref{lem:normalized_words}  the word  $w$  has a prefix $a^k\beta$ or $\alpha^kb$ for some $k\in \mathbb{N}$.
This property  enables us to define a transformation on the set of morphisms from $\mathcal{M} \setminus \langle \varphi_a, \varphi_\alpha\rangle$.
As we will demonstrate later, this transformation is in fact the desired algorithm returning the morphisms $\psi_1, \psi_2, \ldots, \psi_\ell$ mentioned above.
\begin{defi} \label{def:delta}
	Let $w\in  \{a,b,\alpha,\beta\}^*\setminus \{a,\alpha\}^*$ be the normalized name of a morphism $\psi$, i.e., $\psi = \varphi_{w}$. We put
	\[
		\Delta(w) =
			\begin{cases}
				N( w'a^k\beta) & \text{ if \ } w = a^k\beta w', \\
				N(w'\alpha^kb) & \text{ if \ } w = \alpha^kb w'
			\end{cases}
	\]
	and, moreover, $\Delta (\psi) = \varphi_{\Delta(w)}$.
\end{defi}

\begin{example}\label{example:degAnedeg}
	Consider the morphism $\psi = \varphi_w$,  where $w = \beta \alpha aa \alpha$,  and apply repeatedly the transformation $\Delta$ on $\psi$.
	\begin{align*}
		\psi & =  \varphi_{\beta \alpha a a \alpha} \quad \text{ and } \quad  N(\psi)= w =\beta \alpha aa \alpha \\
		\Delta (\psi) & = \varphi_{\alpha a a \alpha \beta} \quad \text{ and } \quad N\bigl(\Delta (\psi)\bigr)= \beta bb \alpha  \alpha \\
		\Delta^2 (\psi) & = \varphi_{ bb \alpha  \alpha \beta} \quad \text{ and } \quad N\bigl(\Delta^2 (\psi)\bigr) = bb \beta  \alpha  \alpha \\
		\Delta^3 (\psi) & = \varphi_{b \beta  \alpha  \alpha b} \quad \text{ and } \quad N\bigl(\Delta^3(\psi)\bigr) = b \beta  \alpha  \alpha b \\
		\Delta^4 (\psi) & = \varphi_{\beta  \alpha  \alpha b b} \quad \text{ and } \quad N\bigl(\Delta^4 (\psi)\bigr)  =  \beta  \alpha  \alpha bb \\
		\Delta^5 (\psi) & = \varphi_{\alpha  \alpha bb \beta} \quad \text{ and } \quad N\bigl(\Delta^5 (\psi)\bigr) =  \alpha  \alpha bb \beta \\
		\Delta^6 (\psi) & =  \Delta^3 (\psi)
	\end{align*}
	In what follows we prove that the five fixed points of morphisms $\Delta (\psi), \Delta^2(\psi), \Delta^3(\psi), \Delta^4(\psi), \Delta^5(\psi)$ are exactly the five derivated words of the fixed point of $\psi$.
\end{example}

\begin{lem}\label{lem:tvar_prefixu}
	Let $\uu$ be a fixed point of a morphism $\psi$ and $ N(\psi) = w \in \{a,b,\alpha,\beta\}^*$ be the normalized name of the morphism $\psi$.
	If one of the following condition is satisfied
	\begin{enumerate}[(i)]
		\item $\uu$ starts with $0$ and $w$ starts with $a$,
		\item $\uu$ starts with $1$ and $w$ starts with $\alpha$,
	\end{enumerate}
	then $w \in \{a, \alpha\}^*$.
\end{lem}
\begin{proof}
	We consider only the case $(i)$, the case $(ii)$ is analogous.
 	Let us assume $w \notin \{a, \alpha\}^*$.
 	According to Lemma \ref{lem:normalized_words}, the word $w$ has a prefix $a^k\beta$, for some $k \geq 1$.
 	Consequently, the morphism $\psi$ equals $\varphi_a^k \circ \varphi_\beta\circ \eta$ for some morphism $\eta$.
 	Any morphism of this form maps $0$ to $1w_1 $ and $1$ to $1w_2$ for some words $w_1$ and $w_2$.
 	Therefore, the fixed point starts with the letter $1$, which is a contradiction.
\end{proof}
The following theorem along with~\Cref{def:delta} provide the algorithm which to a given Sturmian morphism $\psi$  lists  the morphisms fixing the derivated words of the Sturmian word ${\bf u} = \psi({\bf u})$.
\begin{thm}\label{thm:main_result_vetsina}
	Let $\psi \in \langle\varphi_a, \varphi_b, \varphi_\alpha, \varphi_\beta\rangle$ be a primitive morphism and $ N(\psi) = w \in \{a,b,\alpha,\beta\}^* \setminus \{a, \alpha\}^*$ be its normalized name.
	Denote $\uu$ the fixed point of $\psi$.
	Then $\mathbf{x}$ is (up to a permutation of letters) a derivated word of $\uu$ with respect to one of its prefixes if and only if $\mathbf{x}$ is the fixed point of the morphism $\Delta^j(\psi)$ for some $j \geq 1$.
\end{thm}

\begin{proof}
	Denote $\mathbf{x}_j$ the fixed point of $\Delta^j(\psi), j = 1,2,\ldots$ and assume that $v$ is a right special prefix of $\uu$. 
	We will prove that if $|v| = 1$, then $\mathrm{d}_\uu(v) = \mathbf{x}_1$, and if $|v| > 1$, then there is a right special prefix $v'$ of $\mathbf{x}_1$ such that $|v'| < |v|$ and $\mathrm{d}_\uu(v) = \mathrm{d}_{\mathbf{x}_1}(v')$.
	We can repeat this proof for the prefix $v'$ of $\mathbf{x}_1$ and eventually prove that $\mathrm{d}_\uu(v) = \mathbf{x}_j$ for some $j$ and that for any $j$ there is a right special prefix $v$ of $\uu$ so that $\mathrm{d}_\uu(v) = \mathbf{x}_j$.

	Without loss of generality we assume that the normalized name of $\psi$ is $w= a^k\beta z$.
	This means that $\Delta(\psi) = \varphi_{z}\circ \varphi_{a^k\beta}$.

	First we assume $|v| = 1$.
	If $k > 0$, then the first letter of $\uu$ is $1$ which is not a right special factor.
	This implies that $k = 0$.
	Hence we have that $\uu = \varphi_\beta (\uu')$, where $\uu'= \varphi_{z}(\uu)$.
	By \Cref{it:der_of_images_fi_b_1} of \Cref{prop:der_of_images_fi_b} we obtain $\mathrm{d}_\uu(v) = \mathbf{\uu'}$.
	Lemma~\ref{lem:rotace_morfizmu} says the word $\uu'$ is fixed by the morphism $\varphi_{z}\circ \varphi_{\beta} = \Delta(\psi)$, which implies $\uu' = \mathbf{x}_1$.

	Now assume $|v| > 1$.
	If $k = 0$, then by \Cref{it:der_of_images_fi_b_2} of \Cref{prop:der_of_images_fi_b} there is a right special prefix $v'$ of $\uu'= \varphi_{z}(\uu)$ such that $|v'| < |v|$ and $\mathrm{d}_\uu(v) = \mathrm{d}_{\uu'}(v')$.
	Again by Lemma~\ref{lem:rotace_morfizmu} we obtain $\uu' = \mathbf{x}_1$.

	Let $k > 0$.
	For $i =0, 1, \ldots, k$ we define  $\uu^{(i)} = \varphi_{a^{k-i}\beta z}(\uu)$.
 	By Lemma~\ref{lem:tvar_prefixu}, the words $\uu^{(i)}$ all start with the letter $1$.
 	Obviously, $\uu^{(0)} = \uu$ and $\uu^{(i)} = \varphi_a\bigl(\uu^{(i+1)}\bigr)$ for $i =0,1,\ldots, k-1$.
 	By Theorem~\ref{thm:preimages_fi_a}, there are factors $v^{(i)}$ with $i =0, 1, \ldots, k$ such that
 	$v^{(i)}$ is a right special prefix of $\uu^{(i)}$, 
 	\[
 		|v| = |v^{(0)}| \geq |v^{(1)}| \geq |v^{(2)}| \geq \cdots \geq |v^{(k)}|
 	\]
 	and
 	\[
 		\mathrm{d}_{\uu}(v) = \mathrm{d}_{\uu^{(1)}}(v^{(1)}) = \mathrm{d}_{\uu^{(2)}}(v^{(2)}) \cdots = \mathrm{d}_{\uu^{(k)}}(v^{(k)})\, .
 	\]
 	Define $\uu' = \varphi_{z}(\uu)$.
	Then $\uu^{(k)} = \varphi_{\beta z}(\uu) = \varphi_\beta (\uu')$ and by \Cref{it:der_of_images_fi_b_2} of \Cref{prop:der_of_images_fi_b} there is a right special prefix $v'$ of $\uu'= \varphi_{z}(\uu)$ such that $|v'| < |v^{(k)}|$ and $\mathrm{d}_{\uu^{(k)}}(v^{(k)}) = \mathrm{d}_{\uu'}(v')$.
	According to Lemma \ref{lem:rotace_morfizmu}, the word $\uu'$ is fixed by the morphism $\varphi_{z}\circ \varphi_{a^k\beta} = \Delta(\psi)$.
	Thus, we have again proved that there is a prefix $v'$ of $\uu' = \mathbf{x}_1$ such that $|v'| < |v|$ and $\mathrm{d}_\uu(v) = \mathrm{d}_{\uu'}(v')$.
\end{proof}

\begin{remark}
	In Example \ref{example:degAnedeg} we considered the morphism $\psi = \varphi_w$, where $w = \beta \alpha aa \alpha$.
	We have found only five different morphisms $\Delta^i (\psi)$ for $i=1,\ldots,5$.
	The sixth morphism $\Delta^6 (\psi)$ already coincides with $\Delta^3 (\psi)$.
	As it follows from the proofs of Theorems \ref{thm:preimages_fi_b} and \ref{thm:main_result_vetsina}, the fixed points of $\Delta^3 (\psi)$, $\Delta^4 (\psi)$ and $\Delta^5 (\psi)$ represent the derivated words of $\uu$ to infinitely many prefixes of $\uu$.
	Whereas the fixed point of $\Delta (\psi)$ or $\Delta^2(\psi)$ is a derivated word of $\uu$ to only one prefix of $\uu$.
\end{remark}

\begin{example}\label{ex:derivated_words_of_Fibon}
	As explained in Example~\ref{fibonacci1}, to find the derivated words of the Fibonacci word we consider the morphism $\psi = \tau^2 =\varphi_b\varphi_\beta $.
	We have $\Delta(\psi) = \varphi_\beta\varphi_b$ and $\Delta^2(\psi) = \psi$.
	But these two morphisms are equal up to a permutation of letters, as $E\psi E = \Delta(\psi)$.
	This means that all derivated words of the Fibonacci word are the same and coincide with the Fibonacci word itself.
\end{example}

\subsection{Morphisms with two fixed points}

Let us now consider a Sturmian morphism $\psi$ which has two fixed points.
Let us denote $\uu^{(0)}$  and $\uu^{(1)}$ the fixed points of $\psi$ starting with $0$ and $1$, respectively. 
Clearly, $\psi(0)$ starts with $0$ and $\psi(1)$ with $1$.
Since the morphism $\psi$ has to belong to the monoid $\langle \varphi_a, \varphi_\alpha \rangle$, the transformation $\Delta$ cannot be applied on it. 
However, we will show that there is a morphism from $\langle \varphi_a, \varphi_\beta\rangle$ (or $ \langle \varphi_b, \varphi_\alpha\rangle$) with a unique fixed point $\mathbf{v}$ such that the set of derivated words of $\uu^{(0)}$ (or $\uu^{(1)}$) equals to $\{{\bf v}\} \cup {\rm Der}(\bf v)$.
And since $\mathbf{v}$ is a fixed point of some morphism from $\langle \varphi_a, \varphi_b, \varphi_\beta, \varphi_\alpha \rangle \setminus \langle \varphi_a, \varphi_\alpha \rangle$, the set ${\rm Der}(\bf v)$ can be described using \Cref{thm:main_result_vetsina}.

Here we give results only for the case when the normalized name $w  \in\{a,\alpha\}^*$ of the morphism begins with $a$. 
The case when the first letter is $\alpha$ is completely analogous.
It suffices to exchange $a \leftrightarrow b$ and $\alpha \leftrightarrow \beta$ in the statements and proofs.

\begin{lem}\label{lem:jinam} 
	Let $w \in\{a,\alpha\}^*$ be the normalized name of a morphism starting with the letter $a$.
	Then the normalized name $N(wb)$ has a prefix $b$ and a suffix $a$, the word $v=b^{-1}N(wb)$ belongs to $\{a,\beta\}^*$,  and $|v|_{\beta} = |w|_{\alpha}$.
\end{lem}

\begin{proof} 
	First, we consider the special case when $w=a^k\alpha^\ell$, with $k \geq 1$ and $\ell \geq 0$.
	By the relation \eqref{eq:relations}, $N(wb)=ba^{k-1}\beta^\ell a$ and the statement is true. 

	Let $w \in\{a,\alpha\}^*$ be arbitrary.
	It can be decomposed to several blocks of the form $a^k\alpha^\ell$ with $k \geq 1$, $\ell \geq 0$.
	Now the proof can be easily finished by induction on the number of these blocks.
\end{proof}

\begin{prop}\label{lem:revers_of_standard}
	Let $w \in\{a,\alpha\}^*$ be the normalized name of a primitive morphism $\psi$ and let $a$ be its first letter.
	\begin{enumerate}[(i)]
		\item Let $\uu$ be the fixed point of $\psi$ starting with $0$.
		Denote $v =b^{-1}N(wb)\in \{a,\beta\}^*$.
		Then $ {\rm Der}(\uu) = \{{\bf v}\} \cup {\rm Der}(\bf v)$, where ${\bf v }$ is the unique fixed point of the morphism $\varphi_v$\,.

		\item Let $\uu$ be the fixed point of $\psi$ starting with $1$.
		Put $v ={\rm cyc} (w)$ (see\eqref{eq:def_of_cyc}).
		Then ${\rm Der}(\uu) = {\rm Der}({\bf v})$, where ${\bf v }$ is the fixed point of the morphism $\varphi_v$.
	\end{enumerate}
\end{prop}

\begin{proof}
	Let us start with proving $(i)$.
	Let ${\bf v}$ be the infinite word given by Lemma \ref{lem:prevod_fi_a_na_fi_b}.
	Then
	\[
		\varphi_b({\bf v}) = \uu =  \psi(\uu) = \varphi_w(\uu) =  \bigl(\varphi_w \circ \varphi_b\bigr)({\bf v}) = \varphi_{wb}({\bf v})=  \varphi_{N(wb)}({\bf v}) \,.
	\]
	By definition of $v$ we have $N(wb) = bv$ and thus
	\[
		\varphi_b({\bf v}) = \varphi_{bv}({\bf v}) = \varphi_b \bigl( \varphi_v({\bf v}) \bigr).
	\]
	This implies that ${\bf v} = \varphi_v({\bf v}) \,.$
	Since $v \notin \{ a, \alpha \}^*$, the morphism $\varphi_v$ has a unique fixed point, namely the word $\bf v$.
 	By Theorem \ref{thm:preimages_fi_b}, ${\rm Der}(\uu) = \{{\bf v}\} \cup {\rm Der}(\bf v)$ as stated in~$(i)$.

	Statement~$(ii)$ is a direct consequence of Theorem~\ref{thm:preimages_fi_a} and Lemma~\ref{lem:rotace_morfizmu}.
\end{proof}

%
%
\section{Bounds on the number of derivated words} \label{sec:number_der_wo}
In this section we study the relation between the normalized name $w$ of a primitive  morphism $\psi = \varphi_w$ and the number of distinct  return words to its fixed point. 
We restrict ourselves to the case when $w \notin \{a, \alpha\}^*$, as the case  $w \in \{a, \alpha\}^*$  is treated in the next section.

Theorem \ref{thm:main_result_vetsina} says that  the number of derivated words of $\uu$ cannot exceed the upper bound:
\[
	\text{number of distinct words in the sequence } \left(\Delta^k(w)\right)_{k \geq 1}.
\]
Since the words $\Delta^k(w) \in \{a,b,\alpha,\beta\}^*$ are all of the same length and $\Delta^{k+1}(w)$ is completely determined by $\Delta^{k}(w)$, the sequence $\left(\Delta^k(w)\right)_{k \geq 1}$ is eventually periodic. 

The number of distinct elements in $\left(\Delta^k(w)\right)_{k \geq 1}$ is only an upper bound on the number of derivated words of $\uu$.  
As we have already mentioned in Remark \ref{nerozlisuj}, fixed points of morphisms corresponding to the names $v$ and $F(v)$ coincide up to exchange of letters $0$ and $1$ and hence define the same derivated word. 
On the other hand, if $v$ and $v'$ are normalized names with $|v| = |v'|$ and fixed points of $\varphi_{v}$ and $\varphi_{v'}$ coincide (up to exchange of letters), then either $v'=v$ or $ v'= F(v)$.

First we look at two examples that illustrate some special cases of the general \Cref{prop:upper_bound} on the period and preperiod of the sequence $\bigl(\Delta^k(w)\bigr)_{k\geq 1}$.
\begin{example}\label{exa:pomoc6}
Consider a word $w$ of length $n$ in the form $w=b^{n-2}\beta a$.
The sequence of $\bigl(\Delta^k(w)\bigr)_{k\geq 1}$ is eventually periodic.
Its preperiod equals $n-2$ and is given by the words $b^{n-k} \beta b^{k-2} a$, for $ k=3,4,\ldots,n$.
The period equals $n-1$ and is given by the words $b^{n-k} a\beta b^{k-2} $, for $ k=2,3, \ldots,n$.

Let us stress that for any $v \in  \{a,b,\alpha,\beta\}^*$   the equation  $v' = F(v)$ implies $|v|_a = |v'|_\alpha$ and
$|v|_b = |v'|_\beta$. Since all words $\Delta^k(w)$ we listed above contain one letter $a$ and no letter $\alpha$, we can conclude that  the morphism $\varphi_w$ has $2n-3$ distinct derivated words.
\end{example}

\begin{example}\label{exa:pomoc7}

Consider a normalized name $w$ in which the letter $b$ is missing and $w$ contains all the three remaining letters.
Necessarily $w$ has the form
\[
	\beta^{\ell_1} a^{k_1}\beta^{\ell_2}a^{k_2} \cdots \beta^{\ell_s}a^{k_s}\alpha^j,
\]
where $s\geq 1$, $\ell_i\geq 1$ for all $i= 2,\ldots ,s$ and  $k_i\geq 1$ for all $i=1,2,\ldots ,s-1$ and $j\geq 1$.
It is easy to see that the normalized names of words obtained by repeated application of the mapping $\Delta$ are
\[
	\Delta^{\ell_1}(w)= a^{k_1}\beta^{\ell_2}a^{k_2} \cdots \beta^{\ell_s}a^{k_s}\beta^{\ell_1} \alpha^j\qquad \text{and} \qquad \Delta^{\ell_1+1}(w)= \beta^{\ell_2-1}a^{k_2} \cdots \beta^{\ell_s}a^{k_s}\beta^{\ell_1+1}\alpha^{j-1} b^{k_1}\alpha
\]
We see that the $(\ell_1+1)^{st}$  iteration already contains all four letters.
\end{example}

\begin{prop}\label{prop:upper_bound}
	Let $w \in \{a,b,\alpha,\beta\}^* \setminus \{a,\alpha\}^*$ be the normalized name of a primitive Sturmian morphism $\psi =\varphi_w$.
	Then the sequence $\bigl(\Delta^k(w)\bigr)_{k\geq 1}$ is eventually periodic and:
	\begin{enumerate}[(i)]
		\item If it is purely periodic, then its period is at most $|w|$, otherwise, its period is at most $|w| -1$.
		\item If both $b$ and $\beta$ occur in $w$, then the preperiod is at most $|w|-2$, otherwise the preperiod is at most $2|w|-3$.
	\end{enumerate}
\end{prop}
\begin{proof}

By Lemma \ref{lem:normalized_words}, the word $w$ (and all the elements of the sequence $\bigl(\Delta^k(w)\bigr)_{k\geq 1}$) has the  form  $w =a^i\beta w' $ or  $w =\alpha^ib w' $ for some $ i\geq 0$.
In this proof we distinguish three cases such that exactly one of them is valid for all $\Delta^k(w), k = 1,2, \ldots$ 
The first two cases correspond to the ``periodic'' part of the sequence $\bigl(\Delta^k(w)\bigr)_{k\geq 1}$.

{\bf Case 1}: If $w$ has a suffix $\beta$ or $b$, then the word $\Delta(w)$ equals to $w' a^i\beta $  or  $w' \alpha^ib$ and thus  has again a suffix $\beta$ or $b$.
Indeed, since $N(w) = w$, the words $\alpha a^j \beta$ and $a \alpha^j b$ are not factors of $w$ and so they are not even factors of $w'$.
As the last letter of $w'$ is $b$ or $\beta$, neither $\alpha a^j \beta$ nor $a \alpha^j b$ is a factor of $w' a^i\beta$ and hence $\Delta(w) = N(w' a^i\beta) = w' a^i\beta$.  
This means that for any $k$ the word $\Delta^k(w)$ is just a cyclic shift of $w$ (see\eqref{eq:def_of_cyc}).
Therefore,  $\bigl(\Delta^k(w)\bigr)_{k\geq 1}$ is purely periodic and its period is given by the number of letters $\beta$ and $b$ in $w$ which is clearly at most $|w|$.
Moreover, the word $w$ belongs to the sequence  $\bigl(\Delta^k(w)\bigr)_{k\geq 1}$ and the fixed point $\uu$ of $\psi$ itself is a derivated word of $\uu$.

\medskip

Without loss of generality we assume that $w =a^i\beta w'$; the case of $w =\alpha^ib w'$ can be treated in the same way, it suffices to exchange letters $a \leftrightarrow b$ and $\alpha \leftrightarrow \beta$. 
Denote $p$ the longest suffix of $w$ such that $p \in \{a,\alpha\}^*$. 
It remains to consider only the case of nonempty $p$.

\medskip

{\bf Case 2}: If $p=a^j$ for some $j\geq 1$, then $w'$ has a suffix $ba^j$ or $\beta a^j$.
No rewriting rule from~\eqref{eq:relations} can be applied to $w'a^i\beta$, hence, $\Delta(w)= w'a^i\beta$ has a suffix $\beta$. 
So, we can apply the reasoning from Case 1 on the word $\Delta(w)$ and hence the sequence $\bigl(\Delta^k(w)\bigr)_{k\geq 1}$ is purely periodic.
As $w$ contains at least one letter $a$ as a suffix, the period is shorter than $|w|$ and $w$ itself does not occur in $\bigl(\Delta^k(w)\bigr)_{k\geq 1}$.

\medskip

{\bf Case 3}: Now assume that the letter $\alpha$ occurs in $p$.  
We split this case into three subcases and show that if one of these subcases is valid for a word $\Delta^k(w)$, then this word belongs to the ``preperiodic'' part of $\bigl(\Delta^k(w)\bigr)_{k\geq 1}$.
These three subcases (for word $w$) read:\\
\noindent (i)  $w$ begins with the letter $a$, i.e., $i\geq 1$;\\
\noindent (ii) $w$ has a prefix $\beta$ and $p$ has a factor $\alpha a$;\\
\noindent (iii) $w$ has a prefix $\beta$ and $p = a^j\alpha^s$ for $j\geq 0$ and $s\geq 1$.

\medskip

\textbf{(i)} Since we assume that $\alpha$ occurs in $p$, a suffix of $p$ has a form $\alpha a^t$ for some $t \geq 0$. 
It follows that $w'a^i\beta$, has a suffix $\alpha a^{t+i}\beta$.
After applying the rewriting rules~\eqref{eq:relations} to $w'a^i\beta$ we obtain the normalized name $\Delta(w)$ which has a suffix $b\alpha$.

\textbf{(ii)} A suffix of $w$ can be expressed in the form $\alpha a^r\alpha^sa^t$, where $r\geq 1$  and $s, t\geq 0$.
Therefore $w'\beta$ has a suffix $\alpha a^r\alpha^sa^t\beta$. 
After normalization we get that $\Delta(w)$ has a suffix in the form of $b\alpha^\ell$ for some $\ell\geq 1$.

 \textbf{(iii)} As $w = \beta w'$ has a suffix $\beta a^j\alpha^s$ or $ba^j\alpha^s$, the word $N(w'\beta)$ has a suffix $ \beta \alpha^s$.

\medskip

All the three discussed subcases share the following property: The longest suffix $p'\in \{a,\alpha\}^*$ of the normalized name $v = \Delta(w)$ is of the form $p'=\alpha^m$, for some $m\geq 1$.
It means that Case 3 (ii) is not applicable in the second iteration of $\Delta$.

 By \Cref{lem:normalized_words}, the word $v$ has a prefix $a^n\beta $ or $\alpha^n b, n \geq 0$.

 If the prefix of $v$ is of the form $\alpha^n b$, then the word $\Delta(v) = \Delta^2(w)$ belongs to Case 2. 
 This means that $v$ is the last member of the preperiodic part.

 If the prefix of $v$ is of the form $a^n\beta$, then we must apply either Case 3 (i) or  3 (iii)   which means that $\alpha$ is again a suffix of the word obtained in the next iteration of $\Delta$.

Let us give a bound on the number of times that we have to use Case 3 (i) or 3 (iii) before we reach Case 2.

 If $w$ contains both $\beta$ and $b$, then the number of times of using Case 3 (i) or 3 (iii) is at most the number of letters $\beta$ occurring in $w$ before the first occurrence of $b$.
 Thus there are at most $|w|-2$ such letters since $w$ contains $\beta$, $b$ and $\alpha$.

If $w$ does not contain $b$, then $w$ must contain besides the letters $\beta$ and $\alpha$ also the letter $a$; otherwise the morphism $\varphi_w$ would be not primitive (see \Cref{lem:properties_of_sturm_morph}).
The word $w$ has a form described in Example~\ref{exa:pomoc7} and thus $\Delta^{\ell_1+1}(w)$ contains both letter $b$ and $\beta$ (for the meaning of $\ell_1$ see Example~\ref{exa:pomoc7}).
For this word we can apply the reasoning from the previous paragraph, meaning that after $\ell_1+1$ iterations we need at most $|w| -2$ further iterations before reaching the periodic part of $\bigl(\Delta^k(w)\bigr)_{k\geq 1}$. 
Since $\ell_1\leq |w|-2$, we get that the preperiod is at most $2|w|-3$.
\end{proof}

 Example \ref{exa:pomoc6} illustrates that in the  case that $w$ contains the letters $b$ and $\beta$ the upper bounds on preperiod and period provided by the previous proposition are attained. 
 The following example proves that the bound from Proposition \ref{prop:upper_bound} for $w$ which does not contain both letters $b$ and $\beta$ is attained as well.

\begin{example} 
Let us consider the normalized name $w =\beta^{n-2}a\alpha$. 
It is easy to evaluate iterations of the operator $\Delta$:
\begin{align*}
	\Delta^{n-2}(w) & =  a\beta^{n-2}\alpha \\
	\Delta^{n-1}(w) & =  \beta^{n-2}b\alpha \\
	\Delta^{2n-3}(w) & = b\beta^{n-2}\alpha \\
	\Delta^{2n-2}(w) & = \beta^{n-2}\alpha b \quad  \text{--- the first member of the periodic part of $\bigl(\Delta^k(w)\bigr)_{k\geq 1}$} \\
	\Delta^{3n-4}(w) & = \alpha b\beta^{n-2} \quad  \text{--- the last member of the periodic part of $\bigl(\Delta^k(w)\bigr)_{k\geq 1}$}\\
	\Delta^{3n-3}(w) & = \Delta^{2n-2}(w).
\end{align*}
\end{example}

In \Cref{ex:derivated_words_of_Fibon} we showed that for the Fibonacci word the derivated words to all prefixes coincide.
There are infinitely many words with this property:
\begin{example}
	Consider $w=a^{n-1}\beta$ and the morphism $\psi = \varphi_w$.
	Then  $\Delta(\psi)  = \psi$ and thus the fixed point $\uu$ of $\psi$ is the derivated word to any prefix of $\uu$.
\end{example}

Combining \Cref{prop:upper_bound} and the last two examples we can give an upper and lower bound on the number of distinct derivated words.
\begin{coro}\label{coro:bounds_on_nr_of_der_words} 
Let $w \in \{a,b,\alpha,\beta\}^* \setminus \{a,\alpha\}^*$ be normalized name of a primitive Sturmian morphism $\psi =\varphi_w$  and $\uu$ be  a fixed point of $\psi$. 
Then
\begin{equation}\label{meze}
	1\leq \#\Der(\uu) \leq  3|w| -4\,.
\end{equation}
Moreover, for any length $n\geq 2$ there exist normalized names $w', w'' \in \{a,b,\alpha,\beta\}^* \setminus \{a,\alpha\}^*$ of length $n$ such that
\begin{enumerate}[(i)]
	\item $\varphi_{w'}$ and $\varphi_{w''}$ are not powers of other Sturmian morphisms,
	\item for the fixed points $\uu'$ and $\uu''$ of the morphism $\varphi_{w'}$ and $\varphi_{w''}$,  the lower resp. the upper bound in \eqref{meze} is attained.
\end{enumerate}
\end{coro}

%
%

\section{Standard Sturmian morphisms and their reversals}

In this section we provide precise numbers of distinct derivated words for these three types of morphisms:
\begin{enumerate}
	\item $\psi$ is a standard morphism from  $\mathcal{M}$, i.e. $\psi\in \langle \varphi_b,\varphi_\beta\rangle$,
\item $\psi$ is a standard morphism from  $\mathcal{M}\circ E$, i.e.  $\psi\in \langle \varphi_b,\varphi_\beta\rangle\circ E$,
\item $\psi$ is a morphism from  $\langle \varphi_a,\varphi_\alpha\rangle$.

\end{enumerate}
First we explain the title of this section and the fact that the fourth type of Sturmian morphism, namely a Sturmian morphism from $  \langle \varphi_a,\varphi_\alpha\rangle \circ E$,  is not considered at all.

A standard Sturmian morphism is a morphism fixing some standard Sturmian word. A reversal morphism  $\overline{\psi}$ to a morphism $\psi$ is defined by $\overline{\psi}(0) = \overline{\psi(0)}$  and $\overline{\psi}(1) = \overline{\psi(1)}$.   Since $\varphi_a = \overline{\varphi_b}$ and $\varphi_\alpha = \overline{\varphi_\beta}$, any morphism in $\langle \varphi_a,\varphi_\alpha\rangle$ is just a reversal of a morphism in
$\langle \varphi_b,\varphi_\beta\rangle$.

   Due to the form of the morphisms $\varphi_a$ and $\varphi_\alpha$, any morphism $\eta \in \langle \varphi_a,\varphi_\alpha\rangle$ satisfies that  the letter $0$ is a prefix  of $\eta(0)$ and the letter $1$ is a prefix of $\eta(1)$.   As any   morphism $\xi \in \langle \varphi_a,\varphi_\alpha\rangle \circ E$ can be written in the form $\xi(0)= \eta(1)$  and $\xi(1)= \eta(0)$ for some $\eta \in \langle \varphi_a,\varphi_\alpha\rangle$, the morphism $\xi$   cannot have any fixed point.

\medskip

The normalized name $w$ of a standard morphism from $\mathcal{M}$ is composed of the letters $b$ and $\beta$ only. Thus $\Delta(w) = {\rm cyc}(w)$ (see \eqref{eq:def_of_cyc}).

To describe all standard morphisms we have to take into account also the morphisms of the form $\psi = \varphi_w\circ E$.
In this case $\psi^2 \in \langle \varphi_b, \varphi_\beta \rangle$, in particular $\psi^2 = \varphi_{wF(w)}$.
To describe the derivated words of fixed points of these standard morphisms, we need the notation
\[
	{\rm cyc_F}(w_1w_2w_3\cdots w_n) = w_2w_3\cdots w_nF(w_1) \,.
\]

\begin{prop}\label{cor:standard_sturmian}
	Let $\uu$ be a fixed point of a standard Sturmian morphism $\psi$ which is not a power of any other Sturmian morphism.
	\begin{enumerate}[(i)]
		\item If $\psi = \varphi_w$, then $\uu$ has $|w|$ distinct derivated words, each of them (up to a permutation of letters) is fixed by one of the  morphisms
		\[
			\varphi_{v_0}, \varphi_{v_1} , \varphi_{v_2}, \ldots, \varphi_{v_{|w|-1}}, \quad \text{where } v_k ={\rm cyc}^k(w) \text{ for  } k = 0, 1, \ldots, |w|-1.
		\] \label{cor:standard_sturmian_1}
		\item If $\psi = \varphi_w\circ E$, then $\uu$ has $|w|$ distinct derivated words, each of them (up to a permutation of letters) is fixed by one of the  morphisms
		\[
			\varphi_{v_0}\circ E, \varphi_{v_1} \circ E, \varphi_{v_2}\circ E, \ldots, \varphi_{v_{|w|-1}}\circ E, \quad \text{where } v_k ={\rm cyc}_F^k(w) \text{ for } k= 0, 1, \ldots, |w|-1. \label{cor:standard_sturmian_2}
		\]
	\end{enumerate}
\end{prop}

\begin{proof}  (i)  Since $\psi = \varphi_w$ is  a standard morphism, its normalized name $w$ belongs to $\{b, \beta\}^*$ and $\Delta(w) = {\rm cyc}(w)$.  By Theorem \ref{thm:main_result_vetsina}, all derivated words of $\uu$ are fixed by one of the morphisms listed in (i). We only need to show that  fixed points of the listed morphisms differ. More precisely, we need to show that $v_s \neq v_t$ and $v_s\neq F(v_t)$ for all $0\leq t<s\leq |w|-1$.   Here the assumption that $\psi$ is not a power of any other Sturmian morphism is crucial.

Let us recall simple facts about powers of morphisms:
For any $\ell = 1, 2, \ldots$ and  $u \in \{b,\beta\}^+$ we have
\[
	(\varphi_{u})^{\ell}  =  \varphi_{u^\ell} \ ,  \quad (\varphi_{u} \circ E)^{2\ell} = \varphi_{(uF(u))^\ell} \quad \text{and} \quad (\varphi_{u} \circ E)^{2\ell + 1} = \varphi_{(uF(u))^\ell u} \circ E.
\]
If $\psi = \varphi_w$ is not a power of any Sturmian morphism, we have
\begin{equation}\label{nepower}
w \neq u^{\ell} \ \quad \text{and }\quad w\neq \bigl(uF(u)\bigr)^{k}\quad  \text{for any} \   u \in \{b, \beta\}^+  \ \text{and any }\  \ell, k \in \mathbb{N},  \ell\geq 2, k\geq 1.
\end{equation}

 Lemma \ref{lem:Lyndon2} implies that equation ${\rm cyc}^s(w) = {\rm cyc}^t(w)$ has no solution if $w\neq u^\ell$ and $0\leq t<s\leq |w|-1$.  Therefore all the normalized names $v_0,v_1, \ldots, v_{|w|-1}$ are distinct.

\medskip

Now assume that $v_s = {\rm cyc}^{s}(w) = F\bigl({\rm cyc}^{t}(w) \bigr) = F(v_t)$, where  $0\leq t < s\leq  |w|-1$.

Let $z$ and $p$ be  the words such that ${\rm cyc}^{s}(w) = zp$, where  $|z| = s - t$.
We have $zp = F(p)F(z)$ and by \Cref{lem:words_equations_F} there is $x$ such  that ${\rm cyc}^{s}(w) = zp = x(F(x)x)^i(F(x)x)^jF(x) = (xF(x))^{i + j + 1}$ for some non-negative integers $i, j$.
This implies that there is a factor $y$ of $xF(x)$ such that $|y| = |x|$ and $w = (yF(y))^{i + j + 1}$ which is a contradiction with \eqref{nepower}.

\medskip

(ii) If we apply Theorem \ref{thm:main_result_vetsina} to the morphism $\bigl(\varphi_w\circ E\bigr)^2 = \varphi_{wF(w)}$, we obtain the list of $2|w|$ normalized names ${\rm cyc}^{s}(wF(w))$, with  $s=0,1, \ldots, 2|w|-1$. As ${\rm cyc}^{|w|+i}(wF(w)) = {\rm cyc}^{i}(F(w)w)$,
all the derivated words are given by the fixed points of morphisms
\[
	\varphi_{v_0F(v_0)}, \varphi_{v_1F(v_1)}, \varphi_{v_2F(v_3)}, \ldots, \varphi_{v_{|w|-1}F(v_{|w|-1})}
\]
that are just squares of morphisms listed in Item (ii) of the proposition.
 To finish the proof,  we need to show that the fixed points of the listed morphisms do not coincide nor coincide after exchange of the letters   $0\leftrightarrow 1$.   In other words we need to show
$v_sF(v_s) \neq v_tF(v_t)$ and  $v_sF(v_s) \neq F(v_t)v_t$. \\

Assume the contrary. Then $v_s = v_t$ or $v_s=F(v_t)$ for some $t<s$.  If we put $k=s-t$,   then
$v_s ={\rm cyc}_F^k(v_t)$.  Let $v_t = zp$, where $|z| = k$, then $v_s = pF(z)$. Since the morphism $\psi = \varphi_w\circ E$ is not a power of other morphism we know that
\begin{equation}\label{nepower2}
w\neq \bigl(uF(u)\bigr)^{\ell}u \quad  \text{for any} \   u \in \{b, \beta\}^+  \ \text{and any }\  \ell \in \mathbb{N},  \ell\geq 1.
 \end{equation}

Two cases  $v_s = v_t$  and  $v_s=F(v_t)$ will be discussed separately. 
\begin{itemize}
\item If  $v_s = v_t$, then  $zp = pF(z)$ and   \Cref{lem:words_equations_F} says there is $x$ so that  $v_t = zp = (F(x)x)^{i + j}F(x)$, which contradicts \eqref{nepower2}.
\item If $v_s = F(v_t)$, then $zp = F(p)z$ and by \Cref{lem:words_equations_F} there is $x$ so that $v_t = zp = (F(x)x)^{i + j}F(x)$ which is again a contradiction with \eqref{nepower2}. \qedhere
\end{itemize}
\end{proof}

\begin{prop}\label{prop:number_for_aalpha}
Let $w \in \{\alpha, a\}^*$ be the normalized name of a primitive morphism $\psi$ such that the letter $a$ is a prefix of $w$. 
Moreover, assume that $\psi$ is not a power of any other Sturmian morphism.
\begin{enumerate}[(i)]
    \item The fixed point of $\psi$ starting with $0$ has exactly $1+|w|_\alpha$ distinct derivated words.
 	\item The fixed point of $\psi$ starting with $1$ has exactly $1+|w|_a$ distinct derivated words.
\end{enumerate}
\end{prop}

\begin{proof}  We prove only Item $(i)$, the proof of $(ii)$ is analogous. Let $\uu$ denote the fixed point starting with $0$.

Proposition~\ref{lem:revers_of_standard} says that  we have to count elements in the set $ \{{\bf v}\} \cup {\rm Der}(\bf v)$, where $\bf v$  is a fixed point of $\varphi_v$ with the normalized name $v=b^{-1}N(wb)$.  By Lemma \ref{lem:jinam}, the word $v \in\{a, \beta\}^*$. This property  of  $v$ implies that $\Delta^k (v)$ is equal to some cyclic shift ${\rm cyc}^j(v)$ having a suffix $\beta$.
There are $|v|_\beta$ cyclic shifts of $v$ with this property and hence this number is an upper bound for the period of the sequence  $\bigl(\Delta^k(v)\bigr)_{k\geq1}$.  By Lemma  \ref{lem:jinam},  the normalized name $v$ has a suffix  $a$ and  thus the word $v$ itself  does not appear in $\bigl(\Delta^k(v)\bigr)_{k\geq1}$.   We can conclude that  $\uu$ has at most  $1+|w|_\alpha$ derivated  words.  

For each  $k$ the iteration $\Delta^k (v)$ belongs to $\{a, \beta\}^*$ and consequently $F\bigl(\Delta^i(v)\bigr)$ belongs to $ \{a,b\}^*$.  Therefore,  $\Delta^j(v)\neq F\bigl(\Delta^i(v)\bigr)$ for any pair of positive  integers $i,j$. 

As $\psi$ is not a power of any other morphism, we can use the same technique as in the proof of \Cref{cor:standard_sturmian} to show that  ${\rm cyc}^i(v) \neq {\rm cyc}^j(v)$  for $i, j=1,\ldots, |v|$, $i\neq j$. This means that the  period of the sequence $\bigl(\Delta^k(v)\bigr)_{k\geq1}$   is  indeed equal to $|w|_\alpha$ and its preperiod is zero.
\end{proof}
\section{Comments and conclusions}  \label{sec:comments}

\begin{enumerate}[1.,labelwidth=1em,labelindent=0em,leftmargin=!]

\item

In \cite{AraBru05}, the authors studied derivated words only for standard Sturmian words ${\bf c}(\gamma)$.

However, they did not restrict their study to words fixed by a primitive morphism.

Let us show an alternative proof of their result.

The proof is a direct corollary of our \Cref{thm:preimages_fi_b} and the following result of \cite{BeSe_Lothaire}:

\begin{lem}[{\cite[Lemma 2.2.18]{BeSe_Lothaire}}] \label{zLothaira}

For any irrational $\gamma \in (0,1)$ we have
\[
\varphi_b(\mathbf{c}(\gamma)) = \mathbf{c}\left({\frac{\gamma}{1+\gamma}}\right).
\]
\end{lem}

As we have already mentioned, the authors of \cite{AraBru05} required that any derivated word $\mathrm{d}_\uu(v)$ to a prefix $v$ of a Sturmian word $\uu$ starts with the same letter as the word $\uu$.

By interchanging letters $0\leftrightarrow 1$ in a characteristic word $\mathbf{c}(\gamma)$, we obtain the characteristic word $ \mathbf{c}(1-\gamma)$. If $\gamma <\tfrac12$ , then the continued fraction of $\gamma$ is of the form $ [0, c_1+1, c_2, c_3, \ldots]$ with $c_1>0$ and the continued fraction of $1-\gamma$ equals $[0, 1, c_1, c_2, c_3, \ldots]$. Clearly, $ \Der({\bf c}(\gamma))$ and $ \Der({\bf c}(1-\gamma))$ coincide up to a permutation of letters.
Without loss of generality we state the next theorem for the slope $\gamma <\tfrac12$ only.

\begin{thm}[\cite{AraBru05}]
Let ${\bf c}(\gamma$) be a standard Sturmian word and $\gamma = [0, c_1+1, c_2, c_3, \ldots]$ with $c_1>0$.
Then
\[
\Der({\bf c}(\gamma)) = \left \{{\bf c}(\delta) \colon \delta = [0, c_k+1-i, c_{k+1},c_{k+2}, \ldots] \text{ with } 0 \leq i\leq c_{k}-1 \text{ and } (k,i) \neq (1,0) \right \}.
\]
\end{thm}

\begin{proof}

Let $\delta = [0, d_1+1, d_2, d_3, \ldots]$ with $d_1>0$.
Set $\delta' =\frac{\delta}{1-\delta}$.
It is easy to see that $\delta'= [0,d_1, d_2, d_3, \ldots]$.
Since $\delta' \in (0,1)$ and $\delta = \frac{\delta'}{1+\delta'}$, \Cref{zLothaira} implies that ${\bf c}(\delta) = \varphi_b({\bf c}(\delta'))$.
Applying \Cref{thm:preimages_fi_b} we obtain that $\Der({\bf c}(\delta)) = \{{\bf c}(\delta')\} \cup \Der({\bf c}(\delta'))$.
We have transformed the original task to the task to determine the set of derivated words of the standard sequence ${\bf c}(\delta')$.
If $\delta' < \tfrac12$, i.e., $d_1 > 1$, we repeat this procedure with $\delta'$.
If $d_1 = 1$, i.e., $\delta' > \tfrac12$, we use the fact that $\Der({\bf c}(\delta))$ and $ \Der({\bf c}(1-\delta))$ coincide, and replace $\delta'$ by $1-\delta'$ and repeat the procedure with its continued fraction $[0, d_2+1, d_3, d_4, \ldots]$.

In the terms of corresponding continued fractions, one step of the described procedure can be represented as
\[
[0, d_1+1, d_2, d_3, \ldots] \mapsto \begin{cases}
[0,d_1, d_2, d_3, \ldots] & \text{if } d_1 > 1, \\
[0, d_2+1, d_3, d_4, \ldots] & \text{if }d_1 = 1.
\end{cases}
\]
We conclude that  the set $\Der({\bf c}(\gamma))$ is in the form given in the theorem.
\end{proof}
\item  In case that $\uu$ is a fixed point of a standard Sturmian morphisms, we have determined the exact number of distinct derivated words of $\uu$, see \Cref{cor:standard_sturmian}.
Let us mention that this result can be inferred from \cite{AraBru05}.
We also have provided the exact number of derivated words when $\uu$ is a  fixed point of a Sturmian morphisms which has two fixed  points, see \Cref{prop:number_for_aalpha}.

For fixed points of other Sturmian morphisms we only gave an upper bound on the number of their distinct derivated words, see \Cref{coro:bounds_on_nr_of_der_words}.
To give an exact number, one needs to describe when the normalized name $w \in \{a,b,\alpha, \beta\}^*$ corresponds to  some power of a Sturmian morphism. 
Clearly, $w$ may be a normalized name of a  power of a Sturmian morphism without $w$ being a power of some other word from $ \{a,b,\alpha, \beta\}^*$. For example, if $v=\alpha ba\alpha\alpha = N(v)$, then the normalized name of $v^3$  is  the primitive word $ N(v^3)=\alpha bb\beta\beta\beta ba\beta\beta\beta aa \alpha\alpha$.   

\item The key tool we used to determine  the set  $\Der({\bf u})$  is provided by \Cref{thm:preimages_fi_b,thm:preimages_fi_a}.
We believe that an analogue of these theorems can be found also for Arnoux--Rauzy words over multiliteral alphabet.
For definition and  properties of these words see \cite{Be_survey_corr,GlJu}.

In  \cite{CaLaLe17}, the authors described a  new class of ternary sequences with complexity $2n+1$.
These sequences are constructed from infinite products of  two morphisms.
The structure of their bispecial factors suggests that due to result of \cite{BaPeSt}, any derivated  word of such a word is over a ternary alphabet.
Probably, even for these words  an  analogue of \Cref{thm:preimages_fi_b,thm:preimages_fi_a} can be proved.
Other candidates for generalization of Theorems~\ref{thm:preimages_fi_b} and~\ref{thm:preimages_fi_a} seem to be the infinite words whose language forms tree sets as defined in~\cite{tree_sets}.
\end{enumerate}

\section*{Acknowledgements}

K.M., E.P. and Š.S. acknowledge financial support by the Czech Science Foundation grant GA\v CR 13-03538S. K. M. also acknowledges financial support by the Czech Technical University in Prague grant SGS17/193/OHK4/3T/14. 

\bibliographystyle{siam}
\IfFileExists{biblio.bib}{\bibliography{biblio}}{\bibliography{../../!bibliography/biblio}}

\end{document}